\documentclass[11pt,reqno]{amsart}
\usepackage[top=1.3in,bottom=1.3in,left=1.3in,right=1.3in]{geometry}
\usepackage{color}

\usepackage{amssymb}
\usepackage{amsmath}
\usepackage{amsfonts}
\usepackage{geometry}
\usepackage{graphicx}
\usepackage{mathrsfs,amssymb}
\usepackage{stmaryrd}
\usepackage{mathtools}

\usepackage{hyperref}
\usepackage{cleveref}

\usepackage{enumerate}
\usepackage{enumitem}

 \newtheorem{theorem}{Theorem}[section]
 \newtheorem{proposition}[theorem]{Proposition}
 
 \newtheorem{lemma}[theorem]{Lemma}

 \theoremstyle{definition}

 \theoremstyle{remark}
 \newtheorem{remark}[theorem]{Remark}

\numberwithin{equation}{section}

\renewcommand{\Re}{\operatorname{Re}}

\newcommand{\rg}{\operatorname{rg}}

\newcommand{\R}{\mathbb{R}}
\newcommand{\la}{\lambda}

\newcommand{\mb}[1]{\textbf{#1}}
\newcommand{\bm}[1]{{\bf #1}}

\newcommand{\mc}[1]{\mathcal{#1}}

\newcommand{\B}{\mathbb{B}}
\newcommand{\inner}[1]{\langle #1 \rangle}
\newcommand{\norm}[1]{\| #1 \|}
\newcommand{\hsob}[1]{\dot{H}^{#1}(\R^n)}

\title[]{Global-in-space stability of singularity formation for Yang-Mills fields in higher dimensions}

\author{Irfan Glogi\'c}
\address{Faculty of Mathematics, University of Vienna, Oskar-Morgenstern-Platz 1, 1090 Vienna, Austria}
\email{irfan.glogic@univie.ac.at}

\thanks{The author acknowledges support by the Austrian Science Fund FWF, Projects P 30076 and P 34378.}

\usepackage{kantlipsum}

\begin{document}

\begin{abstract}
	We continue our work \cite{Glo22a} on the analysis of spatially global stability of self-similar blowup profiles for semilinear wave equations in the radial case. In this paper we study the Yang-Mills equations in $(1+d)$-dimensional Minkowski space. For $d \geq 5$, which is the energy supercritical case, we consider an explicitly known equivariant self-similar blowup solution and establish its nonlinear  global-in-space asymptotic stability under small equivariant perturbations. The size of the initial data is measured in terms of, in a certain sense, optimal Sobolev norm above scaling.
	This result complements the existing stability results in odd dimensions, while for even dimensions it is new.  
\end{abstract}

\maketitle
\section{Introduction}
\noindent 
The objects we consider are Yang-Mills fields on the trivial bundle $\R^{1+d} \times SO(d,\R)$. That is to say, we study the $1$-forms\footnote{Throughout the paper, we use the Einstein summation convention, where Greek indices run from $0$ to $d$, while Latin indices go from $1$ to $d$. Also, the indices are raised and lowered with respect to the Minkowski metric $\eta_{\alpha\beta}:=\text{diag} (-1,1,\dots,1)$.}
\begin{equation}
	A=A_\alpha(x) \, dx^{\alpha},  \quad A_\alpha: \R^{1+d} \rightarrow \mathfrak{so}(d,\R),
\end{equation}
that extremize the Yang-Mills action functional
\begin{equation*}
	\mathcal{S}[A]:=\int_{\R^{1+d}} \text{tr} (F_{\alpha\beta}F^{\alpha\beta}),
\end{equation*}
where
$
	F_{\alpha\beta}= \partial_\alpha A_\beta - \partial_\beta A_\alpha + [A_\alpha,A_\beta]
$
 is the curvature 2-form. Here, $\mathfrak{so}(d,\R)$ stands for the Lie algebra of real skew symmetric $d \times d$ matrices endowed with the standard commutator bracket. To solve this variational problem, we derive the associated Euler-Lagrange equations
 \begin{equation}\label{Eq:YM_general}
 	\mathbf{D}^\alpha F_{\alpha\beta}=0, \quad \beta=0,\dots,d,
 \end{equation}
where $	{\bf D}_\alpha:=\partial_\alpha + [A_\alpha,\cdot]$ is the covariant derivative associated with the 1-form $A$.\footnote{For notational convenience, in the rest of the paper we view 1-forms as ordered tuples of matrices.} Equations \eqref{Eq:YM_general} are called the \emph{hyperbolic Yang-Mills equations} and we refer to their solutions as \emph{Yang-Mills connections} (or \emph{fields}) on the trivial bundle $\R^{1+d} \times SO(d,\R)$.

Due to the Lorentzian nature of the Minkowski space, one interprets \eqref{Eq:YM_general} as an evolution system in the variable $x^0$, which we henceforth denote by $t$. Consequently, the initial data consist of a pair $(\tilde{A},\tilde{B})$ of $\mathfrak{so}(d,\R)$-valued 1-forms on $\R^d$. Still, this system is under-determined due to the freedom of the choice of gauge. To remove this ambiguity, we study the Cauchy problem for \eqref{Eq:YM_general} in the so-called \emph{temporal gauge}, i.e., we look for solutions for which $A_0 \equiv 0$. Accordingly, we pose the initial condition
\begin{equation}\label{Eq:YM_gen_init_cond}
	A(0,\cdot)=\tilde{A}, \quad  \partial_t A(0,\cdot)=\tilde{B},
\end{equation}
where we assume that $\tilde{A}_0=\tilde{B}_0\equiv 0$.

The system \eqref{Eq:YM_general} obeys the scaling law
\begin{equation*}
	A(t,x) \mapsto A_{\la}(t,x):=\la^{-1} A(t/\la,x/\la),
\end{equation*}
which leaves the $\dot{H}^{\frac{d}{2}-1}$-norm invariant. This leads to the usual criticality classification with respect to the energy.
Our focus in this paper is on the energy-supercritical case, $d \geq 5$, where the presence of nonlinear interactions of $A$ with itself in \eqref{Eq:YM_general} is heuristically expected to cause finite-time blowup for (at least some) large initial data. That this is indeed the case for $5 \leq d \leq 10$ was shown in 1998 by Cazenave-Shatah-Tahvildar-Zadeh  \cite{CazShaTah98}. Then, for all higher dimensions, the confirmation of this heuristic was provided in 2015 by Bizo\'n-Biernat \cite{BizBie15}, who discovered an explicit $SO(d,\R)$-equivariant self-similar solution to \eqref{Eq:YM_general}. $SO(d,\R)$-equivariant solutions represent a subclass of temporal gauge solutions that have the following form 
\begin{equation}\label{Def:Equiv_ansatz}
	A(t,x)=u(t,|x|)\sigma(x),
\end{equation}
where $\sigma : \R^d \rightarrow \mathfrak{so}(d,\R)^d$ is defined by
\begin{equation*}
	\sigma^{ij}_k(x)=\delta^j_k x^i - \delta^i_k x^j,
\end{equation*}
  and $u(t,\cdot):[0,\infty) \rightarrow \R$ is what we call the \emph{radial profile} at time $t$ of the \emph{equivariant 1-form} $A$. The significance of the ansatz \eqref{Def:Equiv_ansatz} lies in the fact that the system \eqref{Eq:YM_general} thereby reduces to a single $(d+2)$-dimensional semilinear radial wave equation for the profile $u=u(t,r)$ 
\begin{equation}\label{Eq:YM_equiv}
	 \partial_t^2 u - \partial_r^2 u - \frac{d+1}{r}\partial_r u =(d-2)u^2 (3-r^2u)
\end{equation}	 
(see \cite{Dum82} for the original derivation), with the corresponding initial condition 
\begin{equation}\label{Eq:YM_equiv_init_cond}
	u(0,r)=u_0(r), \quad \partial_tu(0,r)=u_1(r),
\end{equation}
which is obtained from the equivariant initial data
\begin{equation}\label{Def:Equiv_init_data}
	\tilde{A}(x) = u_0(|x|)\sigma(x), \quad \tilde{B}(x) = u_1(|x|)\sigma(x).
\end{equation}
The self-similar solution (to \eqref{Eq:YM_equiv}) discovered by Bizo\'n-Biernat, exists in the whole supercritical range $d \geq 5$ and is given by
\begin{equation}\label{Def:BB_sol}
	u_T(t,r):=\frac{1}{(T-t)^2}\phi\left( \frac{r}{T-t} \right), \quad \phi(\rho)=\frac{\alpha(d)}{\rho^2+\beta(d)}, \quad T>0,
\end{equation}
where 
\begin{equation}\label{Def:Alpha_Beta}
	\alpha(d)=2\left(1+\sqrt{\tfrac{d-4}{3(d-2)}} \right) \quad \text{and} \quad \beta(d)=\frac{1}{3}\big(2d-8+\sqrt{3(d-2)(d-4)}\big).
\end{equation}
Note that, via relation \eqref{Def:Equiv_ansatz}, $u_T$ yields for the system \eqref{Eq:YM_general} a self-similar temporal gauge solution 
\begin{equation}\label{Def:BB_sol_vector}
	A_T(t,x):=\frac{1}{T-t}\Phi\left(\frac{x}{T-t}\right), \quad 	\Phi(x)= \phi(|x|)\sigma(x),
\end{equation}
which blows up at the origin as $t \rightarrow T^-$.

\subsection{The main result}
To understand the role of  $u_T$ for generic evolutions of \eqref{Eq:YM_equiv}, Bizo\'n-Tabor \cite{BizTab01,Biz02} performed numerical simulations for $d=5$ with randomly chosen initial data that lead to blowup. What they observed is the following: the rate of blowup is always self-similar with the blowup profile being globally in space given by $u_T$ (i.e., $\phi$); see \cite[Figure 3]{Biz02}.  Consequently, they conjectured that for large data evolutions, $A_T$ (i.e., $\Phi$) is the universal global-in-space blowup profile for \eqref{Eq:YM_general} within the equivariant solution class; see \cite[Conjecture 1]{Biz02}. In this paper we prove this conjecture (together with its higher-dimensional analogue) for initial data that are close to $A_T$. More precisely, we show that for every $d \geq 5$ there is an open set (in a suitably chosen topology) of equivariant initial data around 
\begin{equation}\label{Eq:Init_data_equiv}
	A_1(0,\cdot) = \Phi, \quad  \partial_t A_1(0,\cdot)=\Phi + \Lambda \Phi, \quad \text{where} \quad \Lambda f(x):= x^i \partial_i f(x),
\end{equation}
such that the Cauchy evolution of \eqref{Eq:YM_general}-\eqref{Eq:YM_gen_init_cond} blows up by converging globally in space  to $A_T$ (i.e., to $\Phi$ upon dynamical self-similar rescaling) for some $T$ close to 1. The formal statement is as follows.
\begin{theorem}\label{Thm:Main}
	Let $d \geq 5$. There exist constants $0 < \varepsilon \ll 1, M \gg 1$, and $\omega>0$ such that the following holds. For every equivariant initial data 
	\begin{equation}\label{Eq:initial_data}
		(A(0,\cdot),\partial_t A(0,\cdot))=(\Phi,\Phi+ \Lambda \Phi) + (\varphi_0,\varphi_1)
	\end{equation}
	where $\varphi_0,\varphi_1$ are Schwartz functions for which
	\begin{equation}\label{Eq:Data_smallness}
		\| \varphi_0 \|_{\dot{H}^\frac{d-1}{2} \cap \dot{H}^\frac{d}{2}(\mathbb{R}^d)  }  +
		\| \varphi_1 \|_{
		\dot{H}^\frac{d-3}{2} \cap \dot{H}^{\frac{d}{2}-1}(\mathbb{R}^d) }
	  \leq \frac{\varepsilon}{M},
	\end{equation}
	there exists $T \in [1-\varepsilon,1+\varepsilon]$ and a unique classical solution $A \in C^\infty([0,T)\times \R^d)$ to \eqref{Eq:YM_general} which forms singularity at the origin as $t \rightarrow T^-$. Furthermore, the solution $A$ can be written in the following form
		\begin{equation}\label{Eq:Decomposition}
			A(t,x)= \frac{1}{T-t} \left( \Phi\left(\frac{x}{T-t}\right) + \varphi\left(t,\frac{x}{T-t}\right)\right),
		\end{equation}
	where for $s \in \left[\frac{d-1}{2},\frac{d}{2} \right]$ we have that
	\begin{equation}\label{Eq:varphi_small}
		\| \varphi(t,\cdot) \|_{\dot{H}^s(\R^d)} + \| (1+  \Lambda) \varphi(t,\cdot) + (T-t)\partial_t\varphi(t,\cdot) \|_{\dot{H}^{s-1}(\R^d)} \lesssim (T-t)^{\omega}
	\end{equation}
	\noindent for all $t \in [0,T)$. In particular,  $\varphi(t,\cdot) \rightarrow 0$ in $L^\infty(\R^d)$ as $t \rightarrow T^-$, i.e.,
		\begin{equation}\label{Eq:Unif_conv}
		(T-t)A(t,(T-t)\cdot) \rightarrow \Phi
	\end{equation}
	uniformly on $\R^d$ as $t \rightarrow T^-$.
\end{theorem}

 Several remarks are in order.
\begin{remark}
	Typical functional setting for the analysis of wave equations are the $L^2$-based Sobolev spaces. As the profile $(\Phi,\Phi + \Lambda\Phi)$ belongs to $\dot{H}^s \times \dot{H}^{s-1}(\mathbb{R}^d)$ only for $s$ greater than the critical exponent $ s_c := d/2-1$, we analyze the flow \eqref{Eq:YM_general} in homogeneous Sobolev spaces of supercritical order $s > s_c$. 
	Furthermore, as no space $\dot{H}^s \times \dot{H}^{s-1}(\mathbb{R}^d)$ is invariant under the action of the underlying nonlinear operator (see \eqref{Eq:N_0} below), we  consider the intersection spaces
	\begin{equation}\label{Def:inters_sob}
		\dot{H}^{s_1} \cap \dot{H}^{s_2}(\mathbb{R}^d) \times \dot{H}^{s_1-1} \cap \dot{H}^{s_2-1}(\mathbb{R}^d), \quad s_c < s_1 \leq s_2.
	\end{equation}
	The particular choice of exponents $s_1,s_2$ in \eqref{Eq:Data_smallness} imposes itself naturally, and is driven by our desire to work in the largest space of type \eqref{Def:inters_sob}  that is invariant under the action of the nonlinear operator. Our result is in this sense optimal, and it should be contrasted with the choice of a highly suboptimal topology in our preceding work \cite{Glo22a}. We will elaborate on this more throughout the paper; see Section \ref{Sec:Est_nonlin} in particular.
\end{remark}

\begin{remark}
	The profile decomposition \eqref{Eq:Decomposition} and the estimate \eqref{Eq:varphi_small} together imply that the evolution of the perturbation \eqref{Eq:initial_data}  upon dynamical self-similar rescaling converges back to $(\Phi,\Phi + \Lambda \Phi)$ in $\dot{H}^s \times \dot{H}^{s-1}(\R^d)$ as $t \rightarrow T^-$. This corresponds to what is conventionally meant by the stability of a self-similar solution. 
\end{remark}

\begin{remark}
	We note that \eqref{Eq:Unif_conv} implies what has been observed numerically in \cite{BizTab01,Biz02}; see \cite[Figure 3, Conjecture 1]{Biz02} in particular. Derivation of \eqref{Eq:Unif_conv} from \eqref{Eq:varphi_small} follows from the fact that the space of corotational maps in $\dot{H}^{d/2}(\R^d)$ continuously embeds into $L^{\infty}(\R^d)$;  see Appendix \ref{Sec:L^inf_embedding}.
\end{remark}

\begin{remark}
	Note that the self-similar solution $A_T$, as it blows up at the origin, away from the origin it uniformly approaches a singular static profile. Namely, given $\epsilon>0$, we have that
	\begin{equation}\label{Eq:Conv_static}
		A_T(t,\cdot) \rightarrow \frac{\alpha(d)}{|\cdot|^2}\sigma(\cdot)
	\end{equation}
	uniformly on $\R^d \setminus \B^d_\epsilon$ as $t \rightarrow T^-$. It is natural to ask as to whether the solution $A$ given in \eqref{Eq:Decomposition} exhibits a similar limiting behavior.
	Although we prove global uniform convergence  in the rescaled variables \eqref{Eq:Unif_conv}, convergence in this ``microscopic" scale does not lead to insight into the limiting blowup profile  in the physical variable $x$. To conclude that the solution $A(t,\cdot)$ does stay pointwise close to the static profile in \eqref{Eq:Conv_static} (which is indeed the best one can hope for) one would need to control the evolution of the perturbation $\varphi(t,\cdot)$ in the critical space $\dot{H}^{s_c} \times \dot{H}^{s_c-1}(\R^d)$. This however, necessitates a different approach to the one we take in this paper. In particular, one would need Strichartz estimates in similarity variables, which is the route we will undertake elsewhere.
\end{remark}

 \subsection{Related results and discussion}\label{Sec:Related_results}
The origins of the Yang-Mills equations belong to particle physics and stem from the desire to describe the weak and strong interactions of elementary particles using non-abelian Lie groups.  There is a large number of works on the rigorous mathematical treatment of the Yang-Mills equations. Here we give a short and non-inclusive overview of the results that are relevant for our context; for more extensive literature review with historical developments related to local well-posedness and singularity formation, see, e.g., \cite{Tho05,BizTab01,Glo22b} and references therein. 

In the energy critical case, $d=4$, numerical evidence for the existence of finite-time blowup was produced by Bizo\'n-Tabor in \cite{BizTab01}. The mechanism they observed is typical for problems of critical type, and represents the so-called bubbling off of a static profile. Later on, the existence of such blowup mechanism was rigorously proved by Krieger-Schlag-Tataru \cite{KriSchTat09}, and Rapha\"el-Rodnianski \cite{RapRod12} who even proved stability of the solution they constructed.

In the energy supercritical case, $d\geq 5$, exhibiting blowup is easier, as self-similar solutions exist. The first construction goes back to Cazenave-Shatah-Tahvildar-Zadeh \cite{CazShaTah98} who used a variational argument to construct a self-similar solution to \eqref{Eq:YM_equiv} for $d \in \{ 5,7,9 \}$. They then used this to exhibit for $d \in \{ 6,8,10 \}$ a singular traveling wave, which represents a non-equivariant temporal gauge blowup solution for \eqref{Eq:YM_general}. What appears to be the solution they produced for $d=5$, was found later on in closed form by Bizo\'n \cite{Biz02}. Around the same time, upon performing numerical simulations, Bizo\'n-Tabor conjectured in \cite{BizTab01} (see \cite{Biz02} for the precise formulation) that Bizo\'n's explicit self-similar solution is, in fact, the universal blowup profile for generic large equivariant data evolutions. Later on, Bizo\'n-Biernat \cite{BizBie15} discovered the higher-dimensional analogue of this solution, and implicitly conjectured that, analogous to the case $d=5$, it describes the universal blowup mechanism for generic large data evolutions.
 
In terms of rigorous results on the stability of the solutions above, the first work is by Donninger \cite{Don14a} for $d=5$. His result relies on the stability analysis framework he developed together with Sch\"orkhuber-Aichelburg \cite{Don11,DonSchAic12} for wave maps (see also his preceding works \cite{Don10a,Don10}), and represents the nonlinear asymptotic stability of Bizo\'n's solution under small equivariant perturbations inside the backward lightcone of the blowup point (this corresponds to showing \eqref{Eq:Unif_conv} for the unit ball $\B^d$). This result is, however, conditional on a certain spectral assumption, which was later rigorously proved by Donninger-Costin-Huang and the author \cite{CosDonGloHua16}. Then, the extension of \cite{Don14a} to all higher odd dimensions, together with the resolution of the underlying spectral problem, was done by the author in \cite{Glo22b}. In addition to this, Donninger-Ostermann \cite{DonOst21} showed that if the initial data are highly localized and regular, then stability holds in domains that strictly contain the backward lightcone, and include regions of the space-time that go even beyond the time of blowup. 

 \subsection{Comments on the proof of the main result}
 Our proof relies on a novel stability analysis framework we put forward in \cite{Glo22a}.  This approach utilizes a global coordinate frame given by  similarity variables that are posed on the whole space $\R^d$. By this, one can keep track of the evolution of perturbations of self-similar profiles globally on $\R^d$. Consequently, in contrast to the approaches mentioned above, one obtains stability along horizontal and spatially global time slices $\{ t \} \times \R^d$. Furthermore, as a byproduct, one gets that blowup can not happen outside the origin.
 In what follows, we give a brief outline of the proof of Theorem \ref{Thm:Main}, and we point out along the way the main differences and improvements brought about in this paper. 
 
 The bulk of our work concerns the radial semilinear wave equation \eqref{Eq:YM_equiv} in $n:=d+2$ dimensions, and the stability analysis of the corresponding solution \eqref{Def:BB_sol}. At the end, we use the equivalence of Sobolev norms of the profile $u$ and the corresponding equivariant 1-form $A$, to obtain Theorem \ref{Thm:Main}.
The proof starts off by passing to   similarity variables
  \begin{equation*}\label{Def:Simil_var_intro}
 	\tau= \tau(t):=\ln \left(\frac{T}{T-t} \right),  \quad \rho=\rho(t,r):=\frac{r}{T-t}.
 \end{equation*}
The advantage of this coordinate system is that the problem of stability of finite-time self-similar blowup becomes the one of the asymptotic stability of a static profile. What is more, due to the semilinear character of the underlying equation, the new problem can be approached through the standard spectral stability method. The main disadvantage, though, is that in the new variables the self-adjoint structure of the problem is lost. Therefore, in order to establish well-posedness and perform the linearized spectral analysis, instead of relying on the existing theory of wave equations, one ought to develop new tools that account for the non-self-adjoint structure.

First, we establish the fundamental result of the linear theory. Namely, in Proposition \ref{Prop:Free_semigroup} we show that the linear evolution in similarity variables is well-posed in the radial intersection Sobolev spaces
\begin{equation}\label{Eq:Inters_Sob_intro}
	\mc H^{s_1,s_2}:= \dot{H}_{r}^{s_1} \cap \dot{H}_{r}^{s_2}(\R^n) \times \dot{H}_{r}^{s_1-1} \cap \dot{H}_{r}^{s_2-1}(\R^n), \quad 1< s_1 \leq s_2.
\end{equation}
We emphasize that, as opposed to \cite{Glo22a}, we allow in this paper for a greater range of the upper Sobolev exponent $s_2$, which makes Proposition \ref{Prop:Free_semigroup} the optimal well-posedness result in spaces \eqref{Eq:Inters_Sob_intro}.
Then, to study the linear flow near the static profile, we first establish compactness of the perturbation coming from the linearization; see Lemma \ref{Lem:Compactness}. This, in particular, allows for the reduction of the spectral analysis to the problem of determining the unstable point spectrum only.  Furthermore, in Lemma \ref{Lem:Compactness}, we allow $s_2$ to take fractional values as well, which later on ends up being essential in establishing, in a certain sense, optimal nonlinear stability result.

One of the most difficult aspects of the problem of stability in similarity variables is the spectral analysis of the linearized operator, due to it being genuinely non-self-adjoint. By this we mean that the underlying eigenvalue problem (provably) can not be reduced to a self-adjoint problem, and therefore necessitates new tools developed specifically for that context. As it turns out, this problem can be reduced to the one that corresponds to the local stability analysis (in lightcones), which we already performed in \cite{Glo22b}. Therefore, by adapting the results from \cite{Glo22b}, we show that, modulo the symmetry-induced eigenvalue $\la=1$, our solution is spectrally stable; see Proposition \ref{Prop:Pert_semigroup}. Then, by a spectral mapping theorem for compactly perturbed semigroups, we propagate the spectral stability to linear stability; see Proposition \ref{Prop:Decay_stable}.

To upgrade the linear stability further to nonlinear stability, we need some sort of continuity of the nonlinear operator. This brings us to another novel component of the paper: the continuity analysis of the underlying nonlinear operator in spaces $\mc H^{s_1,s_2}$. This is performed in Section \ref{Sec:Est_nonlin}, where we provide proofs of a series of nonlinear estimates, which in particular yield the space $\mc H^{\frac{n-3}{2},\frac{n}{2}-1}$ as the largest one among \eqref{Eq:Inters_Sob_intro} that guaranties local Lipschitz continuity; see Proposition \ref{Prop:Nonlin_est_N}. At the same time, our proofs provide a recipe for establishing optimal Schauder estimates in spaces $\mc H^{s_1,s_2}$ for a large class of nonlinearities (including the polynomial ones).

With these results at hand, standard dynamical systems theory techniques yield the nonlinear asymptotic stability in similarity variables; see Theorem \ref{Thm:CoMain}. As is usual in the orbital stability arguments, we mod out the instability $\la=1$ by properly choosing the blowup time. Then, in Section \ref{Sec:Upgrade_class} we establish the persistence of smoothness in similarity variables. Finally, in Section \ref{Sec:main_proof} we go back to physical coordinates, and use the equivalence of Sobolev norms of 1-forms and their radial profiles, so as to obtain Theorem \ref{Thm:Main}. 
   
 \subsection{Notation and conventions}

  By $\mathbb{B}_R^d$ we denote the open ball inside $\R^d$, centered at zero, with radius $R$; for the unit ball, we simply write $\B^d$. By $\mc S(\R^d)$ we denote the space of Schwartz functions. We also allow for vector functions in this space by requiring that every component be a standard scalar Schwartz function.
 Given a closed linear operator $(L,\mc D(L))$ on a Banach space $X$,  we denote by $\rho( L)$ the resolvent set of $L$, while $\sigma( L):= \mathbb{C} \setminus \rho( L)$ stands for the spectrum of $ L$, and $\sigma_p( L)$ denotes the point spectrum. Also, for $\la \in \rho( L)$ we use the following convention for the resolvent $R_{ L}(\la):=(\la I -  L)^{-1}$. By $\ker L$ and $\rg L$ we denote respectively the kernel and the range of $L$. Also, we use the convenient asymptotic notation $a \lesssim b$ to say that there is some $C>0$ such that $a \leq Cb$. Furthermore, we write $a \simeq b$ if $a \lesssim b$ and $b \lesssim a.$ We also use the Japanese bracket notation $\inner{x}= \sqrt{1+|x|^2}$.

 \section{Passage to similarity variables}\label{Sec:Simil_var}
 
 For the most part, our work is concerned with the equation \eqref{Eq:YM_equiv} and stability analysis of its solution \eqref{Def:BB_sol}. To start, we first write down the underlying Cauchy problem. For that, we let $n:=d+2$ and introduce a new dependent variable $v=v(t,x):=u(t,|x|)$ for $(t,x) \in \R \times \R^n$. In this way, we arrive at the following Cauchy problem
  \begin{equation}\label{Eq:NLW}
 	\begin{cases}
 		~~\partial^2_t v - \Delta v = (n-4) v^2 (3-|x|^2 v ),\\
 		\smallskip
 		~~v(0,\cdot)=v_0, \\
 		~~\partial_t v(0,\cdot)=v_1,
 	\end{cases}
 \end{equation}
where $v_0=u_0(|\cdot|)$ and $v_1=u_1(|\cdot|)$.
 
 \subsection{Similarity variables}
Given $T>0$, we define the \emph{(global) similarity variables}
 \begin{equation}\label{Def:Simil_var}
 	\tau= \tau(t):=\ln \left(\frac{T}{T-t} \right),  \quad \xi=\xi(t,x):=\frac{x}{T-t},
 \end{equation}
 by means of which the strip $S_T:=[0,T) \times \R^n$ is mapped   into the upper half-space $H_+:=[0,\infty) \times \R^n$. In addition, we define the rescaled dependent variable
 \begin{equation}\label{Def:psi}
 	\psi(\tau,\xi):=(T-t)^2v(t,x)=T^2e^{-2\tau}v(T-Te^{-\tau},Te^{-\tau}\xi).
 \end{equation}
 Consequently, the evolution of $v$ inside $S_T$ corresponds to the evolution of $\psi$ inside $H_+$. Note that derivative operators with respect to $t$ and $x$ in the new variables become
 \begin{equation}\label{Eq:Diff_law}
 	\partial_t = \frac{e^\tau}{T}(\partial_{\tau} + \Lambda), \quad \text{and} \quad     \partial_{x_i}= \frac{e^{\tau}}{T}\partial_{\xi_i},
 \end{equation}
 where the operator $\Lambda$ acts on the spatial variable $\xi$, and is defined in \eqref{Eq:Init_data_equiv}.  Based on \eqref{Eq:Diff_law},
 we get that the semilinear wave equation \eqref{Eq:NLW} transforms into
 \begin{equation}\label{Eq:NLW_sim_var}
 	\big(\partial^2_\tau + 5 \partial_\tau + 2 \Lambda \partial_\tau - \Delta + \Lambda^2 + 5 \Lambda + 6 \big)\psi =(n-4) \psi^2 (3-|\xi|^2 \psi ).
 \end{equation}
 To analyze \eqref{Eq:NLW_sim_var}, we follow our previous works and take up the abstract approach via the semigroup theory. To that end, we first write \eqref{Eq:NLW_sim_var} in a vector form. Namely, we define
 \begin{equation}\label{Def:psi_1}
 	\psi_1(\tau,\xi):=\psi(\tau,\xi), \quad \psi_2(\tau,\xi):=(\partial_\tau + \Lambda + 2)\psi(\tau,\xi),
 \end{equation} 
 and let\footnote{For the sake of readability, when writing inline we slightly abuse the convention and write column vectors in the row form.} $\Psi(\tau):=(\psi_1(\tau,\cdot),\psi_2(\tau,\cdot))$. This yields an evolution equation for $\Psi$
 \begin{equation}\label{Eq:Evol_equ}
 	\Psi'(\tau) = \widetilde{\bm L}_0 \Psi(\tau) + \bm N_0(\Psi(\tau)),
 \end{equation}
 where
 \begin{equation}\label{Def:Wave_oper_sim}
 	\widetilde{\bf{L}}_0= 
 	\begin{pmatrix}
 		-\Lambda -2 &1\\
 		\Delta & -\Lambda - 3
 	\end{pmatrix}
 \end{equation}
 is the wave operator in similarity variables, and for $\mb u = (u_1,u_2)$, the nonlinearity is given by
 \begin{equation}\label{Eq:N_0}
 	\mb N_0(\mb u) = 
 	\begin{pmatrix}
 		0 \vspace{1mm} \\ N_0(\cdot,u_1)
 	\end{pmatrix}
 	\quad \text{for} \quad  N_0(\xi,u_1)=(n-4) u_1^2 (3-|\xi|^2 u_1 ).
 \end{equation}
Also, the initial data become
 \begin{equation*}
 	\bm U_0(T):= 
 	\begin{pmatrix}
 		Tv_0(T\cdot) \\
 		T^2 v_1(T\cdot)
 	\end{pmatrix}.
 \end{equation*}
 Now, since \eqref{Def:BB_sol} solves \eqref{Eq:YM_equiv}, we have that for all $n \geq 7$ the function
 \begin{equation*}
 	\Psi_0:=
 	\begin{pmatrix}
 		\phi_0 \\
 		\phi_1
 	\end{pmatrix}, \quad \text{where} \quad \phi_0=\phi(|\cdot|) \quad \text{and} \quad \phi_1= 2\phi_0 +  \Lambda \phi_0,
 \end{equation*}
 is a static solution to \eqref{Eq:Evol_equ}. Then, to study evolutions of initial data near $\Psi_0$ we consider the perturbation ansatz
 \begin{equation*}
 	\Psi(\tau) = \Psi_0 + \Phi(\tau).
 \end{equation*} 
In this way, we arrive at the central evolution equation of the paper
 \begin{equation}\label{Eq:Vector_pert}
 	 \Phi'(\tau) = \big(\widetilde{\bm L}_0 + \mb V \big) \Phi(\tau) + \bm N(\Phi(\tau)), 
 \end{equation}
 where
 \begin{equation}\label{Def:V}
 	\mb V = 
 	\begin{pmatrix}
 		0 & 0 \vspace{1mm} \\ V & 0
 	\end{pmatrix}
 	\quad \text{for} \quad  V(\xi)=3(n-4)\phi_0(\xi)\big(2-|\xi|^2\phi_0(\xi)^2\big),
 \end{equation}
 and
 \begin{equation}\label{Def:N}
 	\mb N(\mb u) = 
 	\begin{pmatrix}
 		0 \vspace{1mm} \\ N(\cdot,u_1)
 	\end{pmatrix}
 	\quad \text{for} \quad  N(\xi,u_1) = (n-4)u_1^2 \big( 3-3|\xi|^2 \phi_0(\xi) - |\xi|^2u_1 \big).
 \end{equation}
 Furthermore, the initial data are now
 \begin{equation}\label{Eq:SS_initial_data}
 	\Phi(0) = \Psi(0) - \Psi_{0} =
 	\begin{pmatrix}
 		Tv_0(T\cdot) - \phi_0 \\
 		T^2 v_1(T\cdot) - \phi_1
 	\end{pmatrix}
 	=
 	\begin{pmatrix}
 		Tv_0(T\cdot) - v_0 \\
 		T^2 v_1(T\cdot) - v_1
 	\end{pmatrix}
 	+
 	\mb v =: \mb U (\mb v,T),
 \end{equation}
 where, for convenience, we denoted
 \begin{equation}\label{Def:InitCond_v}
 	\mb v = 
 	\begin{pmatrix}
 		v_0 - \phi_0 \\
 		v_1 - \phi_1
 	\end{pmatrix}.
 \end{equation}
To study the Cauchy problem \eqref{Eq:Vector_pert}-\eqref{Eq:SS_initial_data} we need a well-posedness theory for the linearized problem first. For this, we treat first the unperturbed linear problem.

\section{The linear flow in similarity variables}

\noindent The central object of this section is the free wave equation in similarity variables 
\begin{equation}\label{Eq:Free_wave}
	\Psi'(\tau) = \widetilde{\mb L}_0 \Psi(\tau).
\end{equation}
In what follows, we show that the operator $\widetilde{\mb L}_0$, when supplied with a suitable domain, is closable, with the closure generating an exponentially decaying semi-group in intersection spaces \eqref{Def:inters_sob}. First, we introduce the necessary functional setup.

\subsection{Functional setup}
Our constructions rely on the homogeneous Sobolev inner product
\begin{equation*}
	\langle u,v \rangle_{\dot{H}^s(\R^n)}=  \inner{|\cdot|^s \mc Fu,|\cdot|^s\mc Fv}_{L^2(\R^n)},
\end{equation*}
where $s \geq 0$, $u,v \in C^\infty_{c}(\R^n)$, and
the Fourier transform definition we use is
\begin{equation*}
	\mc Fu(\xi) := \frac{1}{(2 \pi)^{n/2}}\int_{\R^n} u(x)e^{-i \xi \cdot x}dx.
\end{equation*}
Consequently, we have the homogeneous Sobolev norm on $C^\infty_{c}(\R^n)$ 
\begin{equation*}
	\norm{u}^2_{\hsob{s}}:= \inner{u,u}_{\hsob{s}}.
\end{equation*}
At certain instances we will use a more general, $L^p$-based homogeneous Sobolev norm
\begin{equation*}
	\| u \|_{\dot{W}^{s,p}(\R^n)} := \| \mc F^{-1}[|\cdot|^s \mc Fu] \|_{L^p(\R^n)},
\end{equation*}
noting that the two norms above coincide for $p=2$.
When working with integer values of the Sobolev exponent $s$, we will often use an equivalent norm defined via derivatives. Namely, we will rely on the fact that given $k \in \mathbb{N}_0$
\begin{equation}\label{Eq:Norms_equiv}
	\norm{u}_{\dot{W}^{k,p}(\R^n)} \simeq \sum_{|\alpha|=k}\norm{\partial^\alpha u}_{L^p(\R^n)} 
\end{equation}
for all $u \in C^\infty_{c}(\R^n)$. Now, given $s_1,s_2$ with $0 \leq s_1 \leq s_2$ we define the following inner product on the test space $C^\infty_{c}(\R^n)$
\begin{equation*}
	\inner{u,v}_{\dot{H}^{s_1} \cap  \dot{H}^{s_2}(\R^n)  } := \inner{u,v}_{\dot{H}^{s_1}(\R^n)} + \mb{1}_{(0,\infty)}(s_2-s_1)\inner{u,v}_{\dot{H}^{s_2}(\R^n)},
\end{equation*} 
with the corresponding norm $\| \cdot \|_{\dot{H}^{s_1} \cap \dot{H}^{s_2}(\R^n)}$. Note that we also allow for $s_2=s_1$, in which case we have  $\| \cdot \|_{\dot{H}^{s_1} \cap \dot{H}^{s_2}(\R^n)} = \| \cdot \|_{\dot{H}^{s_1} (\R^n)}$.
Furthermore, if $1 \leq s_1 \leq s_2$ then for $\mb u := (u_1,u_2)$ and $\mb v := (v_1,v_2)$, both of which belong to $C^\infty_{c}(\R^n) \times C^\infty_{c}(\R^n)$, we let
\begin{equation}\label{Def:Inner_prod_H}
	\inner{\mb u,\mb v}_{\mc H^{s_1,s_2}} := \inner{u_1,v_1}_{\dot{H}^{s_1}\cap\hsob{s_2}} + \inner{u_2,v_2}_{\dot{H}^{s_1-1}\cap\hsob{s_2-1}}.
\end{equation}
Consequently,  we define the space $\mc H^{s_1,s_2}$ as the completion of the radial test space $C^{\infty}_{c,r}(\R^n) \times C^{\infty}_{c,r}(\R^n)$ under the norm defined by the inner product \eqref{Def:Inner_prod_H}. 
For simplicity, we do not explicitly indicate the dependence of $\mc H^{s_1,s_2}$ on the  spatial dimension, as we always assume it is denoted by $n$. In the case $s_2=s_1$, we write $\mc H^{s_1}:=\mc H^{s_1,s_2}$.

\subsection{Existence of the free semigroup on $\mc H^{s_1,s_2}$}\label{Sec:Free_semigroup}

\noindent  After introducing the relevant functional spaces, we now turn to showing that the free operator $\widetilde{\mb L}_0$ generates a semigroup on $\mc H^{s_1,s_2}$. First, we supply $\widetilde{\mb L}_0$ with a domain
\begin{equation*}
	\mc D(\widetilde{\mb L}_0) := C^{\infty}_{c,r}(\R^n) \times C^{\infty}_{c,r}(\R^n).
\end{equation*}
The rest of this section is devoted to proving the following fundamental result.
\begin{proposition}\label{Prop:Free_semigroup}
	Suppose that 
	\begin{equation}\label{Eq:Conds_free_semigroup}
		n \geq 3, \quad 1 < s_1 < \tfrac{n}{2}, \quad s_1 \leq s_2.
	\end{equation}
	Then the operator $\widetilde{\bf L}_0 : \mc D(\widetilde{\bf L}_0) \subseteq \mc H^{s_1,s_2} \rightarrow \mc H^{s_1,s_2}$ is closable, and its closure $(\bm L_0, \mc D(\bm L_0))$ generates a strongly continuous semigroup $({\bf S}_0(\tau))_{\tau\geq 0}$ of bounded operators on $\mc H^{s_1,s_2}$. Furthermore, the semigroup obeys the growth estimate
	\begin{equation}\label{Eq:S_0_decay}
		\norm{{\bf S}_0(\tau){\bf u}}_{\mc H^{s_1,s_2}} \leq e^{(\frac{n}{2}-2-s_1)\tau} \norm{{\bf u}}_{\mc H^{s_1,s_2}}
	\end{equation}
	for ${\bf u} \in \mc H^{s_1,s_2}$ and $\tau \geq 0$.
\end{proposition}

\begin{remark}
	Note that this proposition is a slight variation of \cite[Proposition 4.1]{Glo22a}. The operators $\widetilde{\mb L}_0$ in these two contexts differ by a unit shift (which is reflected in \eqref{Eq:S_0_decay}). The only additional difference is in the range of $s_2$. To account for this, minimal adjustments are needed; we therefore provide a short proof that relies heavily on \cite{Glo22a}.
\end{remark}

\begin{proof}
	As we are in a Hilbert space setting, we resort to the Lumer-Phillips theorem. Fix $n \geq 3$, and let $s_1,s_2$ satisfy \eqref{Eq:Conds_free_semigroup}. First of all, by
	 \cite[Corollary 4.3]{Glo22a}, we have that $( \widetilde{\mb L}_0,\mc D(\widetilde{\mb L}_0))$ is closable, and its closure $(\mb L_0,\mc D(\mb L_0))$ satisfies
	\begin{equation*}
		\Re \inner{\mb L_0 \mb u,\mb u}_{\mc H^{s_1,s_2}} \leq \left(\frac{n}{2}-2-s_1\right)\inner{\mb u, \mb u}_{\mc H^{s_1,s_2}}
	\end{equation*}
for ${\bf u} \in \mc H^{s_1,s_2}$.
Then, based on the critical Sobolev embedding 
\begin{equation*}
	\left\| u \right\|_{L^p(\R^n)} \lesssim \norm{u}_{\dot{H}^{s}(\R^n)}	,
\end{equation*}
where $0 < s < \frac{n}{2}$ and $p=p(s)$ is defined by the scaling condition $\frac{n}{p}=\frac{n}{2}-s$, we have that
\begin{equation*}
	\| \mb u \|_{L^{p(s_1)}(\R^n) \times L^{p(s_1-1)}(\R^n)} \lesssim \| \mb u \|_{\mc H^{s_1,s_2}}
\end{equation*}
for all $\mb u \in C^{\infty}_{c,r}(\R^n) \times C^{\infty}_{c,r}(\R^n)$. Consequently, we have a continuous embedding
\begin{equation}\label{Eq:Embedding_into_Lp}
	\mc H^{s_1,s_2} \hookrightarrow L^{p(s_1)}(\R^n) \times L^{p(s_1-1)}(\R^n).
\end{equation}
This, in particular, allows us to think of elements of $\mc H^{s_1,s_2}$ as pairs of pointwise almost everywhere defined functions on $\R^n$.
Furthermore, by means of this embedding we get that \cite[Lemma 4.5]{Glo22a} holds in this context as well. More precisely, by emulating the proof of \cite[Lemma 4.5]{Glo22a}, where we use \eqref{Eq:Embedding_into_Lp} instead of the $L^\infty$-embedding \cite[(4.7)]{Glo22a}, we get that if $\mb u = (u_1,u_2) \in C_r^\infty(\R^n) \times C_r^\infty(\R^n)$ such that given $\alpha \in \mathbb{N}_0^n$
	\begin{equation}\label{Eq:Decay_assumpt}
	|\partial^{\alpha}u_1(x) | \lesssim \inner{x}^{s_1-\frac{n}{2}-|\alpha|-1} \quad \text{and} \quad |\partial^{\alpha}u_2(x) | \lesssim \inner{x}^{s_1-\frac{n}{2}-|\alpha|-2}
\end{equation}
for all $x \in \R^n$, then  $\bm u \in \mc D(\bm L_0)$ and $\bm L_0 \bm u=\widetilde{\bm L}_0 \bm u$.  Finally, from \cite[Lemma 4.6]{Glo22a} we get that
$\rg  \big((\frac{n}{2}-2) \bm I - \bm L_0 \big) = \mc H^{s_1,s_2}$, and we are done.

\end{proof}
 
\section{Stability analysis in similarity variables}

This section is the heart of the paper. In it we prove a small data - global existence result for the Cauchy problem \eqref{Eq:Vector_pert}-\eqref{Eq:SS_initial_data}. More precisely, we show that there is a suitable choice of space $\mc H^{s_1,s_2}$ such that for any smooth and small enough $\mb v$, there is $T$ close to 1 for which \eqref{Eq:Vector_pert}-\eqref{Eq:SS_initial_data} admits a global classical solution that decays exponentially in $\mc H^{s_1,s_2}$. For this, we treat the linearized problem first.

\subsection{Compactness of the linear perturbation on $\mc{H}^{s_1,s_2}$}\label{Sec:Perturbed_semigroup}

\noindent To understand the linear flow near the static solution $\Psi_{0}$ we study the perturbed wave equation
\begin{equation*}
		\Phi'(\tau) = \big( \widetilde{\mb L}_0 + \mb V ) \Phi(\tau),
\end{equation*}
where $\mb V$ is given in \eqref{Def:V}.
Under the conditions of Proposition \ref{Prop:Free_semigroup}, the existence of the semigroup in $\mc{H}^{s_1,s_2}$ generated by (the closure of) the operator $\widetilde{\mb L}:=\widetilde{\mb L}_0 + \mb V$, $\mc D(\widetilde{\mb L}):=\mc D(\widetilde{\bf L}_0)$ follows from the boundedness of $\mb V : \mc{H}^{s_1,s_2} \rightarrow \mc{H}^{s_1,s_2} $, which we prove below. In addition to this, we show that whenever $s_1,s_2< \frac{n}{2}$, the operator $\mb V$ is in fact compact. This feature will be important for establishing the linear stability of $\Psi_0$ later on. In the anticipation of future applications, we establish a result for a general class of operators $\mb V$. 
\begin{lemma}\label{Lem:Compactness}
	Suppose that 
	\begin{equation}\label{Eq:Conds_param_compactness}
		n \geq 5, \quad 1 < s_1 < \tfrac{n}{2},  \quad  s_1 \leq s_2.
	\end{equation}
	Furthermore, assume that $V \in C^\infty_{r}(\R^n)$ such that given  $\alpha \in \mathbb{N}_0^n$ 
	\begin{equation*}
		|\partial^\alpha V(x) | \lesssim \inner{x}^{-2-|\alpha|}
	\end{equation*}
	for all $x \in \R^n$.
	Then the mapping
	\begin{equation}\label{Eq:Potential_map}
		\begin{pmatrix}
			u_1\\u_2
		\end{pmatrix}  \mapsto 
		\begin{pmatrix}
			0 & 0 \\ V & 0
		\end{pmatrix} 
		\begin{pmatrix}
			u_1 \\ u_2
		\end{pmatrix} 
	\end{equation}
	defines a bounded linear operator on $\mc H^{s_1,s_2}$. If, in addition, $s_2<\frac{n}{2}$, then the operator is compact. 
\end{lemma}
\begin{proof}
	Assume \eqref{Eq:Conds_param_compactness}. To establish boundedness of \eqref{Eq:Potential_map} on $\mc H^{s_1,s_2}$, we first show that \eqref{Eq:Potential_map} is bounded as a map from $\mc H^{s_1,s_2}$ into $\mc H^{\lfloor s_1 \rfloor ,\lceil s_2 \rceil}$ and then we just compose with the continuous embedding $\mc H^{\lfloor s_1 \rfloor ,\lceil s_2 \rceil} \hookrightarrow \mc H^{s_1,s_2}$; for the proof of this embedding see \cite[Lemma 4.4]{Glo22a}. To show the boundedness part, it is enough to prove that given $\alpha \in \mathbb{N}_0^n$ for which $|\alpha| \in \{ \lfloor s_1 \rfloor-1 , \lceil s_2 \rceil-1 \} $, we have that
	\begin{equation*}
			\| \partial^\alpha (V u)  \|_{L^2(\R^n)} \lesssim \norm{u}_{\dot{H}^{s_1} \cap \hsob{s_2}}
	\end{equation*}
for all $u \in C^{\infty}_{c,r}(\R^n)$. For this, we consider multi-indices $\beta,\gamma$ for which $ \lfloor s_1 \rfloor-1  \leq |\beta| + | \gamma | \leq \lceil s_2 \rceil-1$. If $|\gamma|<s_1$ then we use the decay of $V$ (and its derivatives) and Hardy's inequality to get that
\begin{equation*}
	\| \partial^\beta V  \partial^\gamma u  \|_{L^2(\R^n)} \lesssim \norm{|\cdot|^{-{s_1}+|\gamma|}\partial^\gamma u}_{L^2(\R^n)} \lesssim \norm{u}_{\hsob{s_1}}
\end{equation*}
for all $u \in C^\infty_{c,r}(\R^n)$. If $|\gamma| \geq s_1$, then we simply have that
\begin{equation*}
	\| \partial^\beta V  \partial^\gamma u  \|_{L^2(\R^n)} \lesssim \norm{\partial^\gamma u}_{L^2(\R^n)} \lesssim \norm{u}_{\dot{H}^{s_1} \cap \hsob{s_2}}
\end{equation*}
for all $u \in C^\infty_{c,r}(\R^n)$. This establishes the claim about boundedness on $\mc H^{s_1,s_2}$. 

Now, assume that $s_2<\frac{n}{2}$. To prove compactness of \eqref{Eq:Potential_map} on $\mc H^{s_1,s_2}$, we do the following. First, we use a variant of the Rellich-Kondrachov compactness theorem to show that \eqref{Eq:Potential_map} is compact on $\mc H^k$ for integer $k$ with $1<k < \frac{n}{2}$. Then we use an ``extrapolation" argument to extend this to $\mc H^s$ for all $1 < s < \frac{n}{2}$. Compactness on $\mc H^{s_1,s_2}$ then straightforwardly follows. To carry out this plan, we start by letting $k$ be an integer satisfying $1< k <\frac{n}{2}$. Then, by emulating the proof of \cite[Lemma 5.1]{Glo22a} we get that the set
\begin{equation*}
	K_{\alpha}:=\{ \partial^\alpha (Vu) : u \in C^\infty_{c,r}(\R^n), \norm{u}_{\hsob{k}} \leq 1   \} 
\end{equation*}
is totally bounded in $L^2(\R^n)$ for all multi-indices $\alpha$ of length $k-1$. This shows compactness on $\mc H^k$. Since we already have boundedness on $\mc H^s$ for all $1<s<\frac{n}{2}$, the interpolation result by Cwikel \cite[Theorem 2.1]{Cwi92} yields compactness as well. This finishes the proof.
\end{proof}

Boundedness of $\mb V$ on $\mc H^{s_1,s_2}$ implies that the closure of $\widetilde{\mb L}$ is $\mb L:=\mb L_0 + \mb V$ with $\mc D(\mb L)= \mc D(\mb L_0)$. Furthermore, by the Bounded Perturbation Theorem for semigroups we get that the operator $\mb L$ too generates a semigroup.

\begin{proposition}\label{Prop:S}
	Assume that
		\begin{equation}\label{Eq:Conds_param_semigroup}
		n \geq 7, \quad 1 < s_1 < \tfrac{n}{2},  \quad  s_1 \leq s_2.
	\end{equation}
	 Then the operator $\bm L : \mc D(\bm L) \subseteq \mc H^{s_1,s_2} \rightarrow  \mc H^{s_1,s_2}$ generates a strongly continuous semigroup $(\bm S(\tau))_{\tau \geq 0}$ of bounded operators on $\mc H^{s_1,s_2}$. Furthermore, we have that
	\begin{equation}\label{Eq:Growth_S_pert}
		\norm{{\bf S}(\tau){\bf u}}_{\mc H^{s_1,s_2}} \leq e^{(\frac{n}{2}-2-s_1+ \| \bf V\|)\tau} \norm{{\bm u}}_{\mc H^{s_1,s_2}}
	\end{equation}
	for $\bm u \in \mc H^{s_1,s_2}$ and $\tau \geq 0$. 
\end{proposition}

The estimate \eqref{Eq:Growth_S_pert} is not sharp in general. Therefore, to get better insight into the growth properties of $\mb S(\tau)$ we resort to a spectral mapping theorem that underlies this setting. More precisely, in case $s_2 < \frac{n}{2}$, by Lemma \ref{Lem:Compactness} the operator $\mb V$ is compact, which puts us in the setting of \cite[Theorem B.1]{Glo22b}. Still, to apply this theorem we need to determine the unstable point spectrum of $\mb L$, which leads us to the next section.

\subsection{Spectral analysis of the linearized operator on $\mc H^{s_1,s_2}$}\label{Sec:Spec_analy_L}

The following proposition summarizes the spectral properties of $\mb L$ that will be relevant later on.
\begin{proposition}\label{Prop:Pert_semigroup}
	Assume that
	\begin{equation*}
		n \geq 7 \quad \text{and} \quad \frac{n}{2}-2 < s_1 \leq s_2 < \frac{n}{2}.
	\end{equation*}
Then there exists $\omega \in (0,s_1+2-\frac{n}{2})$ such that
\begin{equation*}
		\{ \la \in \sigma(\bm L) : \Re \la \geq -\omega \} = \{ 1 \}.
\end{equation*}
Furthermore, $\la=1$ is a simple eigenvalue of $\bm L$ with an explicit eigenfunction
\begin{equation}\label{Def:g}
	\bm g:=
	\begin{pmatrix}
		g \\ \Lambda g + 3g
	\end{pmatrix},
	\quad \text{where} \quad g(\xi)=\frac{1}{\big(|\xi|^2+\beta(n)\big)^2}.
\end{equation}
\end{proposition}

\begin{proof}
 We heavily rely on two results from \cite{Glo22b}, namely Propositions 3.2 and 3.4. We therefore provide a somewhat short proof which emphasizes the adjustments that are necessary to adapt these results to the global setting of this paper.

Note that the components of $\mb g$ satisfy the decay conditions \eqref{Eq:Decay_assumpt}. Therefore $\mb g \in \mc D(\mb L)$, and consequently $\mb L \mb g = \widetilde{\mb L} \mb g = \mb g $, where the second equality follows by a straightforward calculation. We furthermore have that 
\begin{equation*}
	\{ \la \in \sigma (\mb L)  : \Re \la > \frac{n}{2}-2-s_1 \} = \sigma_p(\mb L),
\end{equation*}
which follows from \cite[Theorem B.1]{Glo22b}. This theorem furthermore implies that the first claim of the proposition is equivalent to
\begin{equation}\label{Eq:Empty_set}
	S_0:=\{ \la \in \sigma_p(\mb L) : \Re \la \geq 0, \la \neq 1  \} = \emptyset. 
\end{equation}
Therefore, in what follows, we prove \eqref{Eq:Empty_set}. We assume the contrary, that there are $\la \in S_0$ and $\mb u=(u_1,u_2) \in \mc D(\mb L)$ such that 
\begin{equation}\label{Eq:Eigenv_vector}
	\mb L \mb u = \la \mb u.
\end{equation}
This, in particular, means that there exists a function $\tilde{u}_1 : [0,\infty) \rightarrow \mathbb{C}$ with the following properties:
\begin{itemize}
	\setlength{\itemsep}{1mm}
	\item[1.] $\tilde{u}_1(|\cdot|)=u_1$ almost everywhere in $\R^n$,
	\item[2.] $\tilde{u}_1 \in C^k(0,\infty)$ and $\tilde{u}_1^{(k)} \in H^{s_1-k}(a,b)$, for any $k \in \{0,1,\dots,\lfloor s_1 \rfloor \}$ and $0<a<b$,
	\item[3.] $\tilde{u}_1$	satisfies the ODE
	\begin{align}\label{Eq:Eigenv}
		(1-\rho^2)\tilde{u}_1''(\rho)+\left(\frac{n-1}{\rho}-2(\la+3)\rho \right)&\tilde{u}_1'(\rho)		\nonumber
		\\-(\la+2)&(\la+3)\tilde{u}_1(\rho)+V(\rho)\tilde{u}_1(\rho)=0
	\end{align}
weakly on $(0,\infty)$.
\end{itemize}   
 By the analyticity properties of the coefficient functions in \eqref{Eq:Eigenv}, we can further assume that $\tilde{u}_1$ belongs to  $C^\infty(0,1) \cap C^\infty(1,\infty)$ and solves \eqref{Eq:Eigenv} classically on $(0,1) \cup (1,\infty)$. To understand the asymptotic/analytic behavior of $\tilde{u}_1$ near $\rho=0$ and $\rho=1$ we use Frobenius theory, as both of these are regular singular points of \eqref{Eq:Eigenv}. In this way we show that, in fact, $\tilde{u}_1 \in C^\infty[0,1]$.  First we treat the point $\rho=0$.
 The set of Frobenius indices there is $\{ 0,2-n \}$. Therefore, there are two linearly independent classical solutions on $(0,1)$ that have the following expansions there
\begin{equation*}
		\tilde{u}_{1,1}(\rho)=\sum_{n=0}^{\infty}a_n\rho^n \quad \text{and} \quad \tilde{u}_{1,2}(\rho)=C\log(\rho)\tilde{u}_{1,1}(\rho)+\rho^{2-n}\sum_{n=0}^{\infty}b_n\rho^n,
\end{equation*}
 for $a_0=b_0=1$ and  some $C \in \mathbb{C}.$ If $C=0$ then both solutions are analytic at $\rho=0$, and, being their linear combination, $\tilde{u}_1$ is also analytic there. If $C\neq 0$, then the fact that $\tilde{u}_{1,2}(|\cdot|) \notin L^{p(s_1)}(\B^n_{\varepsilon}) $ for $0<\varepsilon<1$, implies, according to \eqref{Eq:Embedding_into_Lp}, that 
 $\tilde{u}_1$ is a constant multiple of $\tilde{u}_{1,1}$ near zero. Therefore, $\tilde{u}_1 \in C^\infty[0,1)$.
 
 Now we treat the point $\rho=1$. The set of Frobenius indices there is $\{ 0,\frac{n-5}{2}-\lambda \}$, and we therefore distinguish two cases. First, we assume that $\frac{n-5}{2}-\lambda \in \mathbb{N}_0$. Then, there is a pair of linearly independent classical solutions to \eqref{Eq:Eigenv} on $(0,1)$ that have the following expansions there
  \begin{equation*}
	\tilde{u}_{1,1}(\rho)=(1-\rho)^{\frac{n-5}{2}-\la} \sum_{n=0}^{\infty}a_n(1-\rho)^n \quad \text{and} \quad \tilde{u}_{1,2}(\rho)=C\log(1-\rho)\tilde{u}_{1,1}(\rho)+\sum_{n=0}^{\infty}b_n(1-\rho)^n
\end{equation*}
for $a_0=b_0=1$ and some $C\in \mathbb{C}.$ If $C=0$ then both solutions are analytic near $\rho=1$ and so must be $\tilde{u}_1$. Let now $C\neq 0$. We claim that $\tilde{u}_{1,2}^{(\lfloor s_1 \rfloor ) } \notin H^{\{s_1\}}(\varepsilon,1) $. Indeed, if $n$ is odd then $\{s_1\}>\frac{1}{2}$ and therefore $H^{\{s_1\}}(\varepsilon,1) \hookrightarrow L^\infty(\varepsilon,1)$, but $\tilde{u}_{1,2}^{(\lfloor s_1 \rfloor ) }$ is unbounded near $\rho=1$; if $n$ is even then $\tilde{u}_{1,2}^{(\lfloor s_1 \rfloor ) }$ does not belong to $L^2(\varepsilon,1)$ and therefore not to $H^{\{s_1\}}(\varepsilon,1)$ either. Consequently, based on the second of the three properties of $\tilde{u}_1$ listed above, we conclude that $\tilde{u}_1$ is a multiple of $\tilde{u}_{1,1}$ and is therefore analytic at $\rho=1$. It remains to treat the case $\frac{n-5}{2}-\lambda \notin \mathbb{N}_0$. The two normalized Frobenius solutions in that case are given by 
  \begin{equation*}
	\tilde{u}_{1,1}(\rho)=(1-\rho)^{\frac{n-5}{2}-\la} \sum_{n=0}^{\infty}a_n(1-\rho)^n \quad \text{and} \quad \tilde{u}_{1,2}(\rho)=\sum_{n=0}^{\infty}b_n(1-\rho)^n,
\end{equation*}
for $a_0=b_0=1$. Similarly to above, we conclude that $\tilde{u}_{1,1}^{(\lfloor s_1 \rfloor ) } \notin H^{\{s_1\}}(\varepsilon,1) $, and thereby infer that $\tilde{u}_1$ is a multiple of $\tilde{u}_{1,2}$. In conclusion, we have that $\tilde{u}_1 \in C^\infty[0,1]$. This however, is not possible, as proven in \cite[Proposition 3.2]{Glo22b}. We hence arrive at a contradiction, thereby proving the first claim of the proposition.

It remains to show that $\la=1$ is a simple eigenvalue of $\mb L$. For this, we define the Riesz projection 
\begin{equation*}
	\mb P:= \frac{1}{2\pi i}\int_{\gamma}\mb R_{\mb L}(\la) \, d\la,
\end{equation*}
where $\gamma$ is a positively oriented circle centered at $1$ with radius $r_\gamma < 1$. From the definition of $\bm P$ we have that $\bm P \bm g = \bm g$, and therefore $\inner{\bm g} \subseteq \rg \bm P$. To prove the reversed inclusion, we do the following. Since the operator $\mb V$ is compact, we use \cite[Theorem B.1]{Glo22b} to conclude that  $\rg \bm P$ is finite dimensional. Then, based on this, it is enough to show that there does not exist $\bm u \in \mc D(\bm L)$ such that $(\bm I - \bm L)\bm u = -\bm g$. By applying the reasoning from above, this follows from the proof of \cite[Proposition 3.4, i]{Glo22b}. 
\end{proof}

\subsection{Linear stability of the static profile}

Since the operator $\mb L$ has no unstable spectral points other than $\la=1$, the spectral mapping theorem \cite[Theorem B.1]{Glo22b} readily implies the linear stability ``orthogonal" to the unstable eigenspace.

\begin{proposition}[Exponential decay on the stable subspace]\label{Prop:Decay_stable}
	Assume that
\begin{equation*}
		n \geq 7, \quad \frac{n}{2}-2 < s_1\leq s_2 < \frac{n}{2},
\end{equation*}
	 and let $\omega$  be the one from Proposition \ref{Prop:Pert_semigroup}. Then there exists $C \geq 1$ such that
	\begin{equation}\label{Eq:Decay_stable}
		\norm{\bm S(\tau)(\bm I-\bm P)\bm u}_{\mc H^{s_1,s_2}} \leq C e^{-\omega\tau}\norm{(\bm I-\bm P)\bm u}_{\mc H^{s_1,s_2}}
	\end{equation}
	for all $\bm u \in \mc H^{s_1,s_2}$ and all $\tau \geq 0$.
\end{proposition}

\subsection{Estimates of the nonlinearity in the intersection spaces}\label{Sec:Est_nonlin}

This section is devoted to proving local Lipschitz continuity of the nonlinear operator $\mb N$ in $\mc H^{s_1,s_2}$, for suitably chosen Sobolev exponents $s_1,s_2$. We start with proving two auxiliary lemmas.

\begin{lemma}\label{Lem:Nonlin_est_1}
 Let $n \geq 5$. Suppose that $f \in C^\infty(\R^n)$ such that given $\alpha \in \mathbb{N}_0^n$
	\begin{equation}\label{Decay_f}
		|\partial^\alpha f (x)| \lesssim 
		\langle x \rangle^{-|\alpha|} 
	\end{equation}
for all $x \in \R^n$. Then, whenever $s_1,s_2$ satisfy
$
1 \leq s_1 \leq \frac{n}{2}-1 \leq s_2,
$
we have that
	\begin{equation}\label{Eq:Nonlin_1}
		\| f u_1u_2 \|_{\dot{H}^{s_1-1} \cap \dot{H}^{s_2-1} (\R^n)} \lesssim  \prod_{i=1}^{2} \| u_i \|_{\dot{H}^{s_1} \cap \dot{H}^{s_2} (\R^n)}
	\end{equation}
for all $u_1,u_2 \in C^\infty_{c}(\R^n)$.
\end{lemma}

\begin{proof}
To establish this lemma we rely on two fundamental results from harmonic analysis. The first one is Sobolev embedding, i.e., that
\begin{equation}\label{Eq:Sobol_embed}
	\norm{u}_{\dot{W}^{t,q}(\R^n)} \lesssim \norm{u}_{\dot{W}^{s,p}(\R^n)}  
\end{equation}
for all $u \in C^{\infty}_{c}(\R^n)$ whenever $1 < p < q < \infty$  and $s,t \geq 0$ obey the scaling condition  $s-\frac{n}{p} = t - \frac{n}{q}$; see, e.g., \cite[p.~335]{Tao06}. The second result is the generalized Hardy-Rellich inequality, i.e., that
\begin{equation}\label{Eq:Hardy-Relich}
	\left\|\frac{1}{|\cdot|^k}u \right\|_{L^p(\R^n)} \lesssim \norm{u}_{\dot{W}^{k,p}(\R^n)}
\end{equation}
for all $u \in  C^{\infty}_{c}(\R^n)$ whenever $1 < p < \infty$, $k \in \mathbb{N}_0$, and $k  < \frac{n}{p}$; see, e.g., \cite[Corollary 14]{DavHin98}.

Now, given $i \in \{ 1,2\}$, by Sobolev embedding \eqref{Eq:Sobol_embed}, equivalence \eqref{Eq:Norms_equiv}, condition \eqref{Decay_f}, and H\"older's inequality,  we have that 
\begin{align}\label{Eq:Prod_est}
	\| f u_1u_2 \|_{\dot{H}^{s_i-1}(\R^n)} 
	&\lesssim \| f u_1u_2 \|_{\dot{W}^{\lfloor s_i \rfloor,p}(\R^n)} \nonumber \\
	&\lesssim \sum_{\sum_{j=0}^{2}|\alpha_j| =\lfloor s_i \rfloor}\| \partial^{\alpha_0}f \, \partial^{\alpha_1}u_1 \, \partial^{\alpha_2}u_2 \|_{L^p(\R^n)} \nonumber \\
	& \lesssim \sum_{\sum_{j=0}^{2}|\alpha_j| =\lfloor s_i \rfloor} \left \| \frac{1}{|\cdot|^{k_1}} \partial^{\alpha_1} u_1 \right \|_{L^{p_1}(\R^n)} \left \| \frac{1}{|\cdot|^{k_2}} \partial^{\alpha_2} u_2 \right \|_{L^{p_2}(\R^n)} 
\end{align}
for all $u_1,u_2 \in C^\infty_{c}(\R^n)$, whenever 
\begin{equation}\label{Cond1}
	k_1,k_2 \in \mathbb{N}_0,  \quad k_1 + k_2 \leq |\alpha_0|, 
\end{equation}
\begin{equation}\label{Cond2}
2 \leq p_1,p_2 < \infty, \quad \frac{n}{p_1} + \frac{n}{p_2} = \frac{n}{2}+1-\{s_i\}.
\end{equation}
Now, we derive \eqref{Eq:Nonlin_1} from \eqref{Eq:Prod_est}  by a combination of \eqref{Eq:Hardy-Relich} and \eqref{Eq:Sobol_embed}. For this, it is enough to show that given $|\alpha_0|+|\alpha_1|+|\alpha_2|= \lfloor s_i \rfloor$ there is a choice of $k_1,k_2,p_1,p_2$ which, in addition to  \eqref{Cond1} and \eqref{Cond2}, satisfy the following  conditions
\begin{equation}\label{Cond3}
	k_1 < \frac{n}{p_1}, \quad k_2 < \frac{n}{p_2},
\end{equation}
\begin{equation}
	s_1 \leq k_1 + |\alpha_1| + \frac{n}{2}-\frac{n}{p_1} \leq s_2,
\end{equation}
\begin{equation}\label{Cond5}
	s_1 \leq k_2 + |\alpha_2| + \frac{n}{2}-\frac{n}{p_2} \leq s_2.
\end{equation}
To show that making such choice is always possible, we assume without loss of generality that $ | \alpha_1 | \leq |\alpha_2 |$ and then distinguish several cases.

First, if $|\alpha_2 | \geq  \frac{n}{2}-1 $, we can take $k_1=k_2=0$, $p_2=2$, and $p_1=\frac{n}{1-\{s_i\}}$. All five conditions \eqref{Cond1}-\eqref{Cond5} manifestly hold.

Now, assume $|\alpha_2 | < \frac{n}{2}-1$ and $| \alpha_1 | + |\alpha_2 | \geq  \lfloor s_1 \rfloor +1$. Then necessarily $i=2$ and we can therefore take $k_1 = k_2 = 0$, $\frac{n}{p_2}=1+|\alpha_2| $ and $p_1$ defined by \eqref{Cond2}. Straightforward check shows that all five conditions \eqref{Cond1}-\eqref{Cond5} are satisfied.

Finally, we assume that $| \alpha_1 | + |\alpha_2 | \leq \lfloor s_1 \rfloor$. If $i=1$ or $\lfloor s_1 \rfloor=\lfloor s_2 \rfloor$, then an admissible choice is $k_1 = \lfloor s_1 \rfloor- | \alpha_1 | - |\alpha_2 | $, $k_2=0$, $\frac{n}{p_2}=1+|\alpha_2| $, and $p_1$ defined by \eqref{Cond2}. Straightforward check shows that \eqref{Cond1}-\eqref{Cond5} hold. If $i=2$ and $\lfloor s_1 \rfloor < \lfloor s_2 \rfloor$, then it is enough to choose $k_1 = \lfloor s_1 \rfloor- | \alpha_1 | - |\alpha_2 | $, $k_2=1$, $\frac{n}{p_2}=\frac{n}{2}+|\alpha_2|-\lfloor s_1 \rfloor-\{ s_2\}  $, and $p_1$ defined by \eqref{Cond2}.
\end{proof}

We will also need the following result.

\begin{lemma}\label{Lem:Nonlin_est_2}
	If $1 \leq s_1 \leq \frac{n-3}{2} \leq s_2$ then
	\begin{equation*}
		\| |\cdot|^2 u_1 u_2 u_3 \|_{\dot{H}^{s_1-1} \cap \dot{H}^{s_2-1} (\R^n)} \lesssim \prod_{i=1}^{3}\| u \|_{\dot{H}^{s_1} \cap \dot{H}^{s_2} (\R^n)}
	\end{equation*}
	for all $u_1,u_2,u_3 \in C^\infty_{c,r}(\R^n)$.
\end{lemma}

\begin{proof}
To prove this lemma, we crucially rely on a Sobolev embedding for weighted $L^\infty$-norms of derivatives of radial functions, i.e., \cite[Proposition B.1]{Glo22a}, which we, for convenience, copy here. Namely, for $n \geq 2$, $\frac{1}{2} < s <\frac{n}{2}$, and $\alpha \in \mathbb{N}_0^n$, we have that
\begin{equation}\label{Eq:Sobolev_inf}
	\| |\cdot|^{\frac{n}{2}-s} \partial^{\alpha}u   \|_{L^\infty(\R^n)} \lesssim \| u \|_{\dot{H}^{|\alpha|+s}(\R^n)}
\end{equation}
for all $u \in C^\infty_{c,r}(\R^n)$. Now, in order to establish the lemma, we rely on a number of auxiliary estimates. First, given $k \in \mathbb{N}_0$ with $k \leq  \lceil \frac{n}{2} \rceil -2$, we have that 
	\begin{equation}\label{Eq:low_k}
			\| |\cdot|^2 u_1 u_2 u_3 \|_{\dot{H}^{k}(\R^n)} 
			\lesssim \| u_1 \|_{\dot{H}^{k+1}(\R^n)}
			\prod_{i=2}^{3}\| u_i \|_{\dot{H}^{\frac{n-3}{2}}(\R^n)}
	\end{equation}
for all $u_1,u_2,u_3 \in C^\infty_{c,r}(\R^n)$. To show this, it is enough to prove that
\begin{equation}\label{Eq:low_k2}
	\| |\cdot|^{2-\ell} \partial^{\alpha_1} u_1 \, \partial^{\alpha_2} u_2 \, \partial^{\alpha_3}  u_3 \|_{L^2(\R^n)} 
	\lesssim \| u_1 \|_{\dot{H}^{k+1}(\R^n)}
	\prod_{i=2}^{3}\| u_i \|_{\dot{H}^{\frac{n-3}{2}}(\R^n)}
\end{equation}
for all $u_1,u_2,u_3 \in C^\infty_{c,r}(\R^n)$, whenever $\ell \in \{0,1,2\}$ and $\ell+|\alpha_1|+|\alpha_2|+|\alpha_3|=k$ with $|\alpha_2| \leq |\alpha_3|$. To this end, we first note that the left-hand side of \eqref{Eq:low_k2} is bounded by  
\begin{align*}
		\| |\cdot|^{\frac{n}{2}-k-1+|\alpha_1|}  \partial^{\alpha_1} u_1 \|_{L^\infty (\R^n)} 
		\| |\cdot|^{\frac{3}{2}+|\alpha_2|} \partial^{\alpha_2} u_2 \|_{L^\infty(\R^n)}
		\| |\cdot|^{-(\frac{n-3}{2}-|\alpha_3|)}  \, \partial^{\alpha_3}  u_3 \|_{L^2(\R^n)}.
\end{align*}
Then, by applying the estimate \eqref{Eq:Sobolev_inf} to the first two terms and Hardy's inequality to the third, we arrive at \eqref{Eq:low_k2}. Consequently, \eqref{Eq:low_k} holds.
Now, by multilinear complex interpolation (see, e.g., \cite{BerLoe76}), we propagate \eqref{Eq:low_k} to fractional Sobolev exponents, i.e., given $1 \leq s \leq \lceil \frac{n}{2} \rceil -1$ we have that
	\begin{equation}\label{Eq:low_s}
	\| |\cdot|^2 u_1 u_2 u_3 \|_{\dot{H}^{s-1}(\R^n)} 
	\lesssim \| u_1 \|_{\dot{H}^{s}(\R^n)}
	\prod_{i=2}^{3}\| u_i \|_{\dot{H}^{\frac{n-3}{2}}(\R^n)}
\end{equation}
for all $u_1,u_2,u_3 \in C^\infty_{c,r}(\R^n)$. To cover the higher values of $s$ we do something different. First, given integer $k \geq  \lceil \frac{n}{2} \rceil - 1$, we prove that
\begin{equation}\label{Eq:high_k}
		\| |\cdot|^2 u_1 u_2 u_3 \|_{\dot{H}^{k}(\R^n)} 
	\lesssim 
	\prod_{i=1}^{3}\| u_i \|_{\dot{H}^{\frac{n-3}{2}} \cap \dot{H}^{k}(\R^n)}
\end{equation}
for all $u_1,u_2,u_3 \in C^\infty_{c,r}(\R^n)$. In line with the reasoning from above, to show this, it is enough to establish the following estimate
\begin{equation}\label{Eq:high_k2}
	\| |\cdot|^{2-\ell} \partial^{\alpha_1} u_1 \, \partial^{\alpha_2} u_2 \, \partial^{\alpha_3}  u_3 \|_{L^2(\R^n)} 
	\lesssim \prod_{i=1}^{3}\| u_i \|_{\dot{H}^{\frac{n-3}{2}} \cap \dot{H}^{k}(\R^n)}
\end{equation}
for all $u_1,u_2,u_3 \in C^\infty_{c,r}(\R^n)$, whenever $\ell \in \{0,1,2\}$ and $\ell+|\alpha_1|+|\alpha_2|+|\alpha_3|=k$ with $|\alpha_1| \leq |\alpha_2| \leq |\alpha_3|$. To this end, we treat two cases. First, if $|\alpha_3| \geq  \lceil \frac{n}{2} \rceil - 1$, then we estimate the left-hand side by
	\begin{align*}
		\| |\cdot|^{1-\frac{\ell}{2}}  \partial^{\alpha_1} u_1 \|_{L^\infty (\R^n)} 
		\| |\cdot|^{1-\frac{\ell}{2}} \partial^{\alpha_2} u_2 \|_{L^\infty(\R^n)}
		\|    \partial^{\alpha_3}  u_3 \|_{L^2(\R^n)},
	\end{align*}
wherefrom the estimate \eqref{Eq:high_k2} follows by \eqref{Eq:Sobolev_inf}. Now, in the complementary case, $|\alpha_3| <  \lceil \frac{n}{2} \rceil - 1$, which is equivalent to $|\alpha_3| \leq  \frac{n-3}{2} $, we estimate the left-hand side of \eqref{Eq:high_k2} by
\begin{align*}
	\| |\cdot|^{a_1}  \partial^{\alpha_1} u_1 \|_{L^\infty (\R^n)} 
	\| |\cdot|^{a_2} \partial^{\alpha_2} u_2 \|_{L^\infty(\R^n)}
	\|  |\cdot|^{-(\frac{n-3}{2}-|\alpha_3|)}  \partial^{\alpha_3}  u_3 \|_{L^2(\R^n)},
\end{align*} 
where $a_1=a_2=\frac{1}{2}\left( \frac{n+1}{2}-|\alpha_3|-\ell\right)$. From here, the estimate \eqref{Eq:Sobolev_inf} and Hardy's inequality lead to \eqref{Eq:high_k}. 
Now, to propagate \eqref{Eq:high_k} to fractional Sobolev exponents, we do the following. First, assume $s > \left \lceil \frac{n}{2}  \right \rceil -1.$
Then the estimate
\begin{equation*}
	\| |\cdot|^2 u_1 u_2 u_3 \|_{\dot{H}^{s-1}(\R^n)} 
	\lesssim 	\| |\cdot|^2 u_1 u_2 u_3 \|_{\dot{H}^{\lfloor s-1 \rfloor }(\R^n)}  + 	\| |\cdot|^2 u_1 u_2 u_3 \|_{\dot{H}^{\lceil s-1 \rceil}(\R^n)}
\end{equation*}
 implies, according \eqref{Eq:low_k} and \eqref{Eq:high_k}, that
\begin{equation}\label{Eq:high_s}
	\| |\cdot|^2 u_1 u_2 u_3 \|_{\dot{H}^{s-1}(\R^n)} 
	\lesssim 
	\prod_{i=1}^{3}\| u_i \|_{\dot{H}^{\frac{n-3}{2}} \cap \dot{H}^{s}(\R^n)}
\end{equation}
for all $u_1,u_2,u_3 \in C^\infty_{c,r}(\R^n)$, whenever 	$s > \left \lceil \frac{n}{2}  \right \rceil -1$.
Finally, from \eqref{Eq:low_s} and \eqref{Eq:high_s} the claim of the lemma follows.
\end{proof}
With the above two lemmas at hand, we readily obtain the central result of this section; for the statement we need the following definition
\begin{equation*}
	\mc B_{\delta}^{s_1,s_2}:= \{ \mb u \in \mc H^{s_1,s_2} : \| \mb u \|_{\mc H^{s_1,s_2}} \leq \delta \}.
\end{equation*}
\begin{proposition}\label{Prop:Nonlin_est_N}
	Let $\delta >0$. Then, given
	\begin{equation*}
		n \geq 5, \quad 1 \leq s_1 \leq \frac{n-3}{2}, \quad \frac{n}{2}-1 \leq s_2,
	\end{equation*}
	we have that
	\begin{equation}\label{Eq:Nonlin_est}
		\| \bm {N} (\bm u) - \bm N (\bm v) \|_{\mc H^{s_1,s_2}} \lesssim \big ( \| \bm u \|_{\mc H^{s_1,s_2}} + \| \bm v  \|_{\mc H^{s_1,s_2}} \big) \| \bm u - \bm v  \|_{\mc H^{s_1,s_2}}
	\end{equation}
	for all $\bm u, \bm v \in \mc B_\delta^{s_1,s_2}$.
\end{proposition}
\begin{proof}
By multilinearity of $\mb N$ and density, it is enough to show \eqref{Eq:Nonlin_est} for $\mb u, \mb v \in C^{\infty}_{c,r}(\R^n) \times C^{\infty}_{c,r}(\R^n) \, \cap \, \mc B_{2\delta}^{s_1,s_2}$. This follows directly from Lemmas \ref{Lem:Nonlin_est_1} and \ref{Lem:Nonlin_est_2}.
\end{proof}
Based on what we have done so far in this section, we infer that the space $\mc H^{\frac{n-3}{2},\frac{n}{2}-1}$ is the largest of the spaces $\mc H^{s_1,s_2}$ that ensures local Lipschitz continuity of the nonlinear operator $\bm N$. In the rest of the paper, for simplicity we denote
\begin{equation*}
	\mc H := \mc H^{\frac{n-3}{2},\frac{n}{2}-1}.
\end{equation*}

\subsection{Existence of global strong solutions for small data}

\noindent With the nonlinear estimate \eqref{Eq:Nonlin_est} at hand, we are in the position to construct strong solutions to \eqref{Eq:Vector_pert}. For convenience, we copy here the underlying Cauchy problem
\begin{equation}\label{Eq:Vector_pert_2}
	\begin{cases}
		~\Phi'(\tau) = \bm L\Phi(\tau) + \bm N(\Phi(\tau)),\\
		~\Phi(0)=\bm U(\bm v, T).
	\end{cases}	
\end{equation}
Since \eqref{Eq:Vector_pert_2} is semilinear, standard techniques from dynamical systems theory will suffice. We have carried out this process for multiple models so far. Therefore, here we introduce the necessary notation, state the result, and then point to a paper that contains the details of the proof. 

First, we use the fact that $\bm L$ generates the semigroup $\bm S(\tau)$, so as to recast \eqref{Eq:Vector_pert_2} into the integral form
\begin{equation}\label{Eq:Duhamel}
	\Phi(\tau)=\mb S(\tau)\mb U(\mb v,T) + \int_{0}^{\tau}\mb S(\tau-s)\mb N(\Phi(s))ds.
\end{equation} 
Then, as $\bm S(\tau)$ decays exponentially on $\ker \mb P$, and $\mb N$ is locally Lipschitz continuous in $\mc H$, we employ a  fixed point argument to show the existence of global decaying solutions for small initial data. To deal with growth caused by the presence of $\rg \mb P$ in the initial data, we use a Lyapunov-Perron type of argument to suppress the growing mode by appropriately choosing the blowup time. To state the result, we need the following Banach space
\begin{equation*}
	\mc X := \{ \Phi \in C([0,\infty),\mc H) : 
	\| \Phi  \|_{\mc X} := \sup_{\tau>0}e^{\omega\tau}\|\Phi(\tau)\|_{\mc H} < \infty
	\}, 
\end{equation*}
where $\omega$ is from Proposition \ref{Prop:Pert_semigroup}. 
We also need the following definition
\begin{equation*}
	\mc X_\delta := \{ \Phi \in \mc X: \| \Phi \|_{\mc X} \leq \delta \}.
\end{equation*}
Now we formulate the central result of this section. For the proof, we point the reader to \cite[Section 8]{Glo22a}, Theorem 8.3  in particular.
\begin{theorem}\label{Thm:CoMain}
	There  exist $0 < \delta \ll 1$ and $N \gg 1$ such that the following holds. If
	\begin{equation}\label{Eq:v_smallness}
		{\bf v} \in \mc H, \quad \text{the components of {\bf v} are  real-valued}, \quad \text{and} \quad \| {\bf v} \|_{\mc H} \leq \tfrac{\delta}{N^2},
	\end{equation}
	then there exist $T \in [1-\frac{\delta}{N}, 1+\frac{\delta}{N}]$ and $ \Phi \in \mc X_{\delta}$ such that \eqref{Eq:Duhamel} holds for all $\tau \geq 0.$ 
\end{theorem}

\subsection{Upgrade to classical solutions}\label{Sec:Upgrade_class}

In this section we establish persistence of regularity: if the initial data are smooth and rapidly decaying then the strong solution constructed in Theorem \ref{Thm:CoMain} is globally smooth, in both space and time. To do this, we analyze the corresponding strong solution to \eqref{Eq:Evol_equ}. More precisely, since $\mb V = \mb L - \mb L_0$ is bounded, the variation of constant formula
	\begin{equation*}
	\mb S(\tau)= \mb S_0(\tau) + \int_{0}^{\tau}\mb S_0(\tau-s)\mb V \mb S(s)ds
\end{equation*}
holds, by means of which we then get from \eqref{Eq:Duhamel} that $\tau \mapsto \Psi(\tau)=\Psi_0 + \Phi(\tau)$ is a global strong $\mc H$-solution to  
\begin{equation}\label{Eq:Duhamel0}
	\Psi(\tau)=\mb S_0(\tau)\mb U_0(T) + \int_{0}^{\tau}\mb S_0(\tau-s)\mb N_0(\Psi(s))ds.
\end{equation}
 In what follows, we show that under regularity and rapid decay assumptions on $\mb v$, the corresponding solution $\Psi(\tau)$, which is a priori only inside $\mc H^{\frac{n-3}{2},\frac{n}{2}-1}$, in fact belongs to $\mc H^{\frac{n-3}{2},k}$ for all $k \geq \frac{n}{2}-1$ . For this, we need to establish several nonlinear estimates first.
\begin{lemma}\label{Lem:Nonlin_est_3}
	Let $n \geq 5$ and $k \in \mathbb{N}, k \geq \lfloor \frac{n}{2} \rfloor$. Then
	\begin{equation}\label{Eq:Nlin_est_1}
		\| u^2 \|_{\dot{H}^{\lfloor \frac{n}{2} \rfloor-1}(\R^n)} \lesssim \| u \|_{ \dot{H}^{\frac{n}{2}-1}(\R^n) } \| u \|_{ \dot{H}^{\lfloor \frac{n}{2} \rfloor}(\R^n) }
	\end{equation}
and
\begin{equation}\label{Eq:Nlin_est_2}
	\|u^2 \|_{\dot{H}^{k}(\R^n)} \lesssim \| u \|_{ \dot{H}^{k}(\R^n) } \| u \|_{\dot{H}^{\frac{n-3}{2}} \cap \dot{H}^{k+1}(\R^n) }
\end{equation}
for all $u \in C^\infty_c(\R^n).$
\end{lemma}
\begin{proof}
	Estimate \eqref{Eq:Nlin_est_1} follows from the fact that for any pair of multi-indices $\alpha,\beta$ with $|\alpha|+|\beta|=\lfloor \frac{n}{2} \rfloor-1$ and $|\alpha| \leq |\beta|$, H\"older's inequality and the critical Sobolev embedding imply that
	\begin{equation*}
		\| \partial^{\alpha}u \, \partial^\beta u  \|_{L^2(\R^n)} \leq \| \partial^{\alpha}u  \|_{L^{p_1}(\R^n)} \|  \partial^\beta u  \|_{L^{p_2}(\R^n)} \lesssim \| u \|_{ \dot{H}^{\frac{n}{2}-1}(\R^n) } \| u \|_{\dot{H}^{\lfloor \frac{n}{2} \rfloor}(\R^n) },
	\end{equation*}
for all $u \in C^\infty_c(\R^n)$, where $p_1,p_2$ are defined by $\frac{n}{p_1}=1+|\alpha|$ and $\frac{n}{p_2}=\frac{n}{2}-\lfloor\frac{n}{2}\rfloor +|\beta|$.
Estimate \eqref{Eq:Nlin_est_2} follows from the combination of the estimate
\begin{equation*}
	\| u^2 \|_{\dot{H}^k(\R^n)} \lesssim \| u \|_{\dot{H}^k(\R^n)} \| u \|_{L^\infty(\R^n)}
\end{equation*}
(see, e.g., \cite[Proposition 3.7]{Tay97}) and the $L^\infty$-embedding
\begin{equation*}
	 \| u \|_{L^\infty(\R^n)} \lesssim \| u \|_{\dot{H}^{\frac{n-3}{2}} \cap \dot{H}^{k+1}(\R^n) }.
\end{equation*}
\end{proof}
Now we establish the necessary nonlinear estimates for the operator $\mb N_0$.

\begin{lemma}
	Let $n \geq 7$.	Define $F(x):=x+x^2$ and let $k \in \mathbb{N}, k \geq \lfloor \frac{n}{2} \rfloor$. Then
	\begin{equation}\label{Eq:Nlin_est_3}
		\| \bm N_0(\bm u) \|_{\mc H^{\frac{n-3}{2},\lfloor \frac{n}{2} \rfloor}} \lesssim  F \big( \| \bm u \|_{\mc H} \big) \,  \| \bm u \|_{\mc H^{\frac{n-3}{2},\lfloor \frac{n}{2} \rfloor}}
	\end{equation}
and 
	\begin{equation}\label{Eq:Nlin_est_4}
	\| \bm N_0(\bm u) \|_{\mc H^{\frac{n-3}{2},k+1}} \lesssim  F \Big( \| \bm u \|_{\mc H^{\frac{n-3}{2},k}} \Big) \,  \| \bm u \|_{\mc H^{\frac{n-3}{2},k+1}}
\end{equation}	
for all $\bm u \in C_{c,r}^\infty(\R^n) \times C_{c,r}^\infty(\R^n)$.
\end{lemma}

\begin{proof}
	To establish \eqref{Eq:Nlin_est_3}, we first observe that from \eqref{Eq:low_k}, or \eqref{Eq:high_k}, depending on the parity of $n$, we have that
		\begin{equation*}
		\||\cdot|^2 u^3 \|_{\dot{H}^{\lfloor \frac{n}{2} \rfloor-1}(\R^n)} \lesssim \| u \|^2_{\dot{H}^{\frac{n-3}{2}} \cap \dot{H}^{\frac{n}{2}-1}(\R^n) } \| u \|_{\dot{H}^{\frac{n-3}{2}} \cap \dot{H}^{\lfloor \frac{n}{2} \rfloor}(\R^n) }
	\end{equation*}
for all $u \in C^\infty_{c,r}(\R^n).$ Then, from this estimate and Lemmas \ref{Lem:Nonlin_est_1}, \ref{Lem:Nonlin_est_2}, \ref{Lem:Nonlin_est_3}, we get \eqref{Eq:Nlin_est_3}. The estimate \eqref{Eq:Nlin_est_4} follows from the same three lemmas together with \eqref{Eq:high_k}.
\end{proof}

We now prove the main result of this section.

\begin{proposition}\label{Prop:Upgrade_to_class}
	If $\bm v$ from Theorem \ref{Thm:CoMain} is such that its components belong to the Schwartz class $\mc S(\R^n)$, then for the corresponding solution $\Phi$ to \eqref{Eq:Vector_pert_2} we have that the map $(\tau,\xi) \mapsto \Psi(\tau)(\xi) = \Psi_0(\xi)+ \Phi(\tau)(\xi)$ belongs to	
	 $C^\infty([0,\infty)\times \R^n) \times C^\infty([0,\infty)\times \R^n)$ and its first component satisfies \eqref{Eq:NLW_sim_var} classically on $[0,\infty)\times \R^n$. 
\end{proposition}

\begin{proof}
	 First, we show that, under the assumptions of the proposition, $\Psi(\tau)$ belongs to $C^\infty(\R^n) \times C^\infty(\R^n)$ for all $\tau \geq 0$. We start by noting that  $\mb U_0(T)$ belongs to $\mc H^{\frac{n-3}{2},k}$ for every integer $k \geq \frac{n}{2}-1$. Furthermore, from Lemmas \ref{Lem:Nonlin_est_1} and \ref{Lem:Nonlin_est_2} we get that the operator $\mb N_0$ is locally Lipschitz continuous on $\mc H^{s_1,s_2}$, whenever $1 \leq s_1 \leq \frac{n-3}{2}$ and $\frac{n}{2}-1 \leq s_2$. Therefore, a straightforward fixed point argument yields for every $k \geq \frac{n}{2}-1$ a local strong $\mc H^{\frac{n-3}{2},k}$-solution to \eqref{Eq:Duhamel0}. This solution, due to uniqueness, must coincide on its interval of existence with the global $\mc H$-solution $\Psi$. Furthermore, the existence of local solutions is of the subcritical type in $\mc H^{\frac{n-3}{2},k}$, i.e., the existence time depends only the norm of the initial data. Therefore, we can use the estimate \eqref{Eq:Nlin_est_3} and Gronwall's inequality to get from \eqref{Eq:Duhamel0} that a solution that is global in $\mc H$ is also global in $\mc H^{\frac{n-3}{2},\lfloor\frac{n}{2}\rfloor}$. Consequently, by \eqref{Eq:Nlin_est_4} we inductively get that for every integer $k \geq \lfloor\frac{n}{2}\rfloor$, $\Psi$ is a global $\mc H^{\frac{n-3}{2},k}$-solution to \eqref{Eq:Duhamel0}. Then, by Sobolev embedding we get that $\Psi(\tau)$ belongs to $C^\infty(\R^n) \times C^\infty(\R^n)$ for all $\tau \geq 0$; see \cite[Lemma 4.4]{Glo22a}. 
	
	Now, we prove regularity in $\tau$.
	Fix an integer $k_0 > \frac{n}{2}+1$. We note that, according to the proof of Proposition \ref{Prop:Free_semigroup}, we have that $\mb U_0(T) \in \mc D(\mb L_0)$ relative to $\mc H^{\frac{n-3}{2},k_0}$. Since, in addition to this, the operator $\mb N_0$ is locally Lipschitz continuous in $\mc H^{\frac{n-3}{2},k_0}$, we have that $\Psi \in C^1([0,\infty),\mc H^{\frac{n-3}{2},k_0})$, $\Psi(\tau) \in \mc D(\mb L_0)$ for all $\tau \geq 0$, and furthermore
	\begin{equation*}
		 	\Psi'(\tau) = \bm L_0 \Psi(\tau) + \bm N_0(\Psi(\tau))
	\end{equation*}
	 holds in the classical (operator) sense in $\mc H^{\frac{n-3}{2},k_0}$; see, e.g., \cite[p.~60, Proposition 4.3.9]{CazHar98}. Additionally, since $\Psi(\tau) \in \mc D(\mb L_0)$ and $\Psi(\tau) \in \mc H^{\frac{n-3}{2},k}$ for every integer $k > \frac{n}{2}+1$, we have that $\mb L_0$ acts classically on $\Psi(\tau)$, i.e., $\mb L_0 \Psi(\tau) = \widetilde{\mb L}_0 \Psi(\tau)$. Consequently, \eqref{Eq:Evol_equ} holds. As $\mc H^{\frac{n-3}{2},k_0}$ is embedded in $L^\infty(\R^n) \times L^\infty(\R^n)$ we conclude that the $\tau$-derivative holds pointwise. Then, by Schwarz's theorem we infer that mixed derivatives of all orders in $\tau$ and $\xi$ exist, and the claim of the proposition follows.

\end{proof}

\section{Back to physical coordinates: Proof of Theorem \ref{Thm:Main}}\label{Sec:main_proof}

Let $\varepsilon:=\frac{\delta}{N}$, where $\delta,N$ are from Theorem \ref{Thm:CoMain}. To deduce Theorem \ref{Thm:Main} from Theorem \ref{Thm:CoMain}, we start with showing that smallness of the perturbation $(\varphi_0,\varphi_1)$ implies smallness of $\mb v$, in their respective norms. For this, we use \cite[Proposition A.5]{Glo22b}, which implies that
\begin{equation}\label{Eq:Equiv_norms}
	\|u(|\cdot|) \sigma   \|_{\dot{H}^k(\R^d)} \simeq \| u(|\cdot|) \|_{\dot{H}^k(\R^{d+2})}	
\end{equation}
for all $u(|\cdot|) \in C^{\infty}_{c,r}(\R^d)$. By \eqref{Eq:Equiv_norms}, from \eqref{Def:InitCond_v} we infer that there is a large enough $M>0$ for which \eqref{Eq:Data_smallness} implies $\| \mb v \|_{\mc H} \leq \frac{\delta}{N^2}$. Then, by Theorem \ref{Thm:CoMain} there exists $T \in [1-\varepsilon,1+\varepsilon]$ and a global strong $\mc H$-solution $\Phi$ to \eqref{Eq:Vector_pert_2} for which 
\begin{equation}\label{Eq:Exp_dec}
	\| \Phi(\tau) \|_{\mc H} \leq \delta e^{-\omega\tau},
\end{equation}
for all $\tau \geq 0$.
Consequently, by denoting $(\tilde{\psi}_1(\tau,|\cdot|),\tilde{\psi}_2(\tau,|\cdot|)) = \Psi(\tau) = \Psi_0 + \Phi(\tau)$, according to Proposition \ref{Prop:Upgrade_to_class},  we get that
\begin{equation*}
	A(t,x):= \tilde{\psi}_1\left( \ln \left(\frac{T}{T-t}\right) , \frac{x}{T-t} \right)\sigma \left( \frac{x}{T-t} \right)
\end{equation*}
belongs to $C^\infty([0,T) \times \R^d)$ and solves the system \eqref{Eq:YM_general} on $[0,T) \times \R^d$ classically. Furthermore, denote $(\tilde{\varphi}_1(\tau,|\cdot|),\tilde{\varphi}_2(\tau,|\cdot|))=\Phi(\tau)$. Then, we have the decomposition
		\begin{equation*}
	A(t,x)= \frac{1}{T-t} \left( \Phi\left(\frac{x}{T-t}\right) + \varphi\left(t,\frac{x}{T-t}\right)\right),
\end{equation*}
where 
\begin{equation}\label{Eq:Equiv_phi}
	\varphi(t,x)= \tilde{\varphi}_1(\ln T - \ln(T-t),|x|)\sigma(x),
\end{equation}
and $\Phi$ is the equivariant 1-form from \eqref{Def:BB_sol_vector} (and not the solution \eqref{Eq:Exp_dec}).
Now, from \eqref{Eq:Equiv_phi}, \eqref{Eq:Exp_dec}, and \eqref{Eq:Equiv_norms} we get that
\begin{equation}\label{Eq:Est1}
	\norm{\varphi(t,\cdot)}_{\dot{H}^{\frac{d-1}{2}} \cap \dot{H}^{\frac{d}{2}} (\R^d)} \lesssim (T-t)^\omega,
\end{equation}
for all $t \in [0,T)$. Similarly, for the time derivative component
\begin{align*}
	\partial_tA(t,x)=\frac{1}{(T-t)^2}&(\Phi + \Lambda \Phi)\left(\frac{x}{T-t}\right) \\ &+ \frac{1}{T-t}\partial_0\varphi\left(t,\frac{x}{T-t}\right)+ \frac{1}{(T-t)^2}(1+\Lambda) \varphi\left(t,\frac{x}{T-t}\right) ,
\end{align*}
we get that
\begin{equation}\label{Eq:Est2}
	\norm{(T-t)\partial_t\varphi(t,\cdot)+(1+\Lambda) \varphi(t,\cdot)}_{\dot{H}^{\frac{d-3}{2}} \cap \dot{H}^{\frac{d}{2}-1} (\R^d)} \lesssim (T-t)^\omega.
\end{equation}
Equation \eqref{Eq:varphi_small} then follows from \eqref{Eq:Est1} and \eqref{Eq:Est2}. The final statement of the theorem follows from \eqref{Eq:Est1} and the $L^\infty$-embedding for critical corotational maps, which we prove in the appendix; see Lemma \ref{Lem:Embedding_inf}.

\appendix

\section{{$L^\infty$}-embedding of equivariant maps in the borderline case $s=\frac d2$}\label{Sec:L^inf_embedding}

\noindent We call a map $U : \R^d \rightarrow \R^d$ \emph{corotational} if there exists $u: [0,\infty) \rightarrow \R$ such that $U(x)= u(|x|) x$ for a.e.~$x\in \R^d$. Also, we call $u$ the \emph{radial profile} of $U$. Now, let us recall the critical Sobolev embedding, i.e., that
\begin{equation*}
	\| f \|_{L^{p}(\R^d)} \lesssim \| f \|_{\dot{H}^s(\R^d)}
\end{equation*}
for all $f \in C_c^\infty(\R^d)$, whenever $0 \leq s < \frac{d}{2}$ and $p= \frac{2d}{d-2s}$. At the endpoint case, $(s,p) = (\frac d2,\infty)$, this estimate breaks down, even for radial functions. As it turns out, corotational maps do obey this borderline embedding.

\begin{lemma} \label{Lem:Embedding_inf}
	We have that
	\begin{equation*}
		\|  U  \|_{L^\infty(\mathbb{R}^d)} \lesssim \|  U \|_{\dot{H}^{\frac{d}{2}}(\mathbb{R}^d)}
	\end{equation*}
	for all smooth and compactly supported corotational maps $U$ on $\R^d$.
\end{lemma}

\begin{proof}
	To establish this lemma, we use the equivalence between Sobolev norms of corotational maps and those of their radial profiles, together with a Sobolev embedding of radial functions into weighted $L^\infty$ spaces; see \cite[Proposition A.5]{Glo22b} and \cite[Proposition B.1]{Glo22a}. By means of the aforementioned two results, we have a one-line proof
	\begin{align*}
		\| U  \|_{L^\infty(\mathbb{R}^d)} \lesssim  \| |\cdot| u(|\cdot|)  \|_{L^\infty(\mathbb{R}^{d+2})} \lesssim \| u(|\cdot|) \|_{\dot{H}^{\frac{d}{2}}(\mathbb{R}^{d+2})}  \simeq \|U \|_{\dot{H}^{\frac{d}{2}}(\mathbb{R}^{d})}.
	\end{align*}
\end{proof}
\begin{remark}
	By using the extension and restriction operators, the same estimate can be proven also for corotational maps on balls $\mathbb{B}_R^d$.
\end{remark}

\bibliography{refs_YM_global_higher_dim}
\bibliographystyle{plain}

\end{document}